\documentclass[11pt,a4paper]{article}

\usepackage[utf8]{inputenc}
\usepackage[T1]{fontenc}
\usepackage[english]{babel}
\usepackage{lmodern}
\usepackage[margin=1.1in, top=1in, bottom=1in]{geometry}
\usepackage{setspace}
\usepackage{amsmath,amssymb,amsthm,mathtools}

\usepackage{enumitem}
\usepackage{booktabs}
\usepackage{graphicx}
\usepackage{verbatim}
\usepackage[hidelinks]{hyperref}
\usepackage[expansion=true]{microtype}

\theoremstyle{plain}
\newtheorem{theorem}{Theorem}[section]
\newtheorem{conjecture}[theorem]{Conjecture}
\newtheorem{proposition}[theorem]{Proposition}
\newtheorem{lemma}[theorem]{Lemma}
\newtheorem{corollary}[theorem]{Corollary}

\theoremstyle{definition}
\newtheorem{definition}[theorem]{Definition}

\newtheorem{hypothesis}[theorem]{Hypothesis}

\theoremstyle{remark}
\newtheorem{remark}[theorem]{Remark}

\DeclareMathOperator{\lcm}{lcm}
\DeclareMathOperator{\Li}{Li}
\DeclareMathOperator{\vp}{v_p}
\renewcommand{\gcd}{\operatorname{gcd}}

\newcommand{\PP}{\mathbb{P}}
\newcommand{\NN}{\mathbb{N}}

\newcommand{\picount}{\pi}

\title{\bfseries Primes in LCM recurrences}
\author{Beno\^{\i}t Cloitre}
\date{}

\makeatletter
\@ifundefined{subjclass}{%
  \newcommand{\subjclass}[2][2020]{%
    \gdef\@subjclass{#2}%
    \gdef\@subjclassyear{#1}%
  }%
}{}
\makeatother


\begin{document}
\maketitle

\begin{abstract}
\noindent
We study an LCM-based analogue of Rowland's GCD-based prime-generating recurrence, introduced by the author in 2008. The multiplicative increments of this sequence are conjectured always to be $1$ or prime, but a complete proof requires a strengthening of Linnik's theorem on the least prime in an arithmetic progression that lies beyond current reach. We develop a Companion--Sieve framework that reduces the conjecture to an equidistribution problem for primes in the progression $-1\bmod p$, and applying the Bombieri--Vinogradov theorem we prove unconditionally that the conjecture holds for a set of integers of asymptotic density~$1$. We also give an effective finite reduction showing that any counterexample beyond a computable threshold involves only large prime factors. A closely related recurrence turns out to encode twin prime pairs through its increment pattern, and we prove a conditional density-$1$ result for it under a prime-index detection hypothesis, using an upper-bound Selberg sieve estimate for twin primes in arithmetic progressions. The analysis also leads to three new conjectures on the distribution of primes in arithmetic progressions.

\medskip

\end{abstract}

\section{Introduction}\label{sec:intro}

\subsection{Rowland's recurrence and the LCM variant}

In 2008, Rowland~\cite{Rowland} showed that the recurrence
\begin{equation}\label{eq:rowland}
a_1=7,\qquad a_n=a_{n-1}+\gcd(n,a_{n-1})\quad (n\ge 2)
\end{equation}
produces first differences $a_n - a_{n-1}$ that are always $1$ or prime. For this particular initial condition the proof is elementary. However, other GCD-based recurrences of the same type remain open, and even for Rowland's sequence the finer analysis of which primes appear and how often requires nontrivial arguments~\cite{CRR11}. The LCM variant studied here raises the difficulty further.

In the same year, we introduced the LCM variant (\href{https://oeis.org/A135504}{OEIS A135504}):
\begin{equation}\label{eq:lcm}
a_1=1,\qquad a_n=a_{n-1}+\lcm(n,a_{n-1})\quad (n\ge 2).
\end{equation}
This is the $K=1$ case of the one-parameter family $a_n=K\,a_{n-1}+\lcm(n,a_{n-1})$ for $K\in\NN^*$, on which we focus throughout except in Section~\ref{sec:twin} where the case $K=2$ is treated. Here the relevant quantity is the multiplicative increment
\begin{equation}\label{eq:bn}
b_n=\frac{a_n}{a_{n-1}}-1=\frac{n}{\gcd(n,a_{n-1})},
\end{equation}
whose first values are
\[
2, 1, 2, 5, 1, 7, 1, 1, 5, 11, 1, 13, 1, 5, 1, 17, 1, 19, 1, 1, 11, 23, \ldots
\]

\begin{conjecture}\label{conj:main}
For every $n\ge2$, $b_n\in\{1\}\cup\PP$.
\end{conjecture}

\noindent Numerical computation confirms this conjecture up to $n=10^6$.

These recurrences have no practical value for producing primes. Their interest lies in the window they open on prime distribution. Simple arithmetic recursions create sequences whose analysis requires, and tests, deep equidistribution results for primes in arithmetic progressions.

\subsection{Why LCM is harder than GCD}

Replacing GCD by LCM changes the problem in two essential ways. First, the sequence exhibits global accumulation. We have $a_n = \prod_{k=2}^n(1+b_k)$, so every past increment is permanently encoded. Second, Ruiz-Cabello~\cite{RuizCabello2017} showed that a pointwise proof of Conjecture~\ref{conj:main} requires, for every prime $p$ and every $k\ge1$, at least $k$ primes $q\equiv -1\pmod{p}$ with $q<p^{k+1}$. This amounts to Linnik's theorem \cite{Linnik44} with constant $L=2$, whereas the best known bound is $L\le 5$ due to Xylouris~\cite{Xylouris2011}, following the earlier work of Heath-Brown~\cite{HB92}. Our approach sidesteps this barrier by working on average rather than pointwise.

\subsection{Main results}

A direct analysis of $a_{n-1}=\prod_{k<n}(1+b_k)$ is difficult because its $p$-adic content receives contributions from all earlier indices, including composite ones, and those contributions depend recursively on the full history of the sequence.

Our method isolates a deterministic sub-product coming from prime indices, where the factors are explicit, and replaces the exact quotient $b_n$ by a tractable majorant. This relaxation loses information, but its local deficits are controlled by counts of primes in the progression $-1\bmod p$.

The counting problem thereby obtained lies in the range of Bombieri--Vinogradov on average. Our main results are as follows.

\begin{theorem}\label{thm:density1-intro}
The set of integers $n$ for which $b_n\in\{1\}\cup\PP$ has asymptotic density $1$.
\end{theorem}

Beyond density, we obtain an effective reduction of the conjecture to large prime factors.

\begin{theorem}\label{thm:finite-intro}
For every $P_0\ge2$, there exists an effectively computable bound $N^*(P_0)$ such that for all $n\ge N^*(P_0)$, any potential counterexample to Conjecture~\ref{conj:main} involves only prime factors larger than $P_0$.
\end{theorem}

Replacing the coefficient $1$ by $2$ in the recurrence leads to a variant with a richer arithmetic structure. Define $x_1=1$ and $x_n=2x_{n-1}+\lcm(n,x_{n-1})$, so that $x_n=x_{n-1}(2+c_n)$ where $c_n=n/\gcd(n,x_{n-1})$ plays the role of $b_n$. The sequence $(c_n)$ begins $2, 3, 1, 1, 1, 7, 2, 1, 1, 11, \ldots$ and is again conjectured to consist of $1$'s and primes. Unlike the original sequence where $b_p=p$ for every prime $p\ge 5$, the $K=2$ variant satisfies $c_p=1$ at certain primes, and this event turns out to be tightly linked to twin primes.

\begin{theorem}\label{thm:twin-intro}
If $(p,p+2)$ are twin primes with $p>5$ and $c_p=p$, then $c_{p+2}=1$. Conversely, under the global conjecture $c_n\in\{1\}\cup\PP$, if $q\ge5$ is prime and $c_q=1$, then $q-2$ is prime. In other words, the forced-$1$ increments at prime indices detect twin prime pairs.
\end{theorem}

The method used for the original sequence adapts to this variant, with the twin-prime loss playing the role of a controlled error term.

\begin{theorem}\label{thm:k2-intro}
Under a hypothesis specifying the behavior of $c_q$ at prime indices, the set of integers $n$ for which $c_n\in\{1\}\cup\PP$ has asymptotic density $1$.
\end{theorem}

\subsection{Organization}

Section~\ref{sec:prelim} collects notation and standard tools. The Companion--Sieve framework is developed in Section~\ref{sec:framework}. The density theorem is proved in Section~\ref{sec:density}, and the effective finite reduction together with the residual obstruction are analyzed in Section~\ref{sec:checkpoints}. The $K=2$ variant and twin primes are treated in Section~\ref{sec:twin}, including the conditional density-$1$ theorem. Section~\ref{sec:conclusion} presents a conditional result under the Elliott--Halberstam conjecture, formulates three conjectures motivated by the residual obstruction, and collects open problems.

\section{Preliminaries}\label{sec:prelim}

\subsection{Notation}

We denote by $\PP$ the set of prime numbers. We write $\vp(n)$ for the $p$-adic valuation of $n$, $\varphi(q)$ for Euler's totient, $\Li(x)=\int_2^x dt/\log t$ for the logarithmic integral, and
\[
\picount(x;q,a)=\#\{p\le x: p\in\PP,\, p\equiv a\pmod q\},
\qquad
E(x;q,a)=\picount(x;q,a)-\frac{\Li(x)}{\varphi(q)}
\]
for the prime counting function in progressions and its error term. As usual, $f\ll g$ means $|f|\le Cg$ for an absolute constant $C$, and $f\asymp g$ means $f\ll g\ll f$.

\subsection{Basic properties of the sequence}

\begin{lemma}\label{lem:basic}
For all $n\ge2$, we have $1\le b_n\le n$ and $b_n\mid n$. Moreover, $a_n$ is even for all $n\ge3$.
\end{lemma}
\begin{proof}
That $b_n\mid n$ follows from $b_n=n/\gcd(n,a_{n-1})$. For parity, note that $a_2=3$ and $a_3=6$. For $n\ge 3$, $a_{n-1}$ is even, hence $\lcm(n,a_{n-1})$ is even, and $a_n=a_{n-1}+\lcm(n,a_{n-1})$ is even by induction.
\end{proof}

\subsection{Tools from analytic number theory}

\begin{theorem}[Siegel--Walfisz; cf.\ {\cite[Ch.~22]{Davenport}}, {\cite[Ch.~5.9]{IwaniecKowalski}}]\label{thm:SW}
For any fixed $A>0$, there exists $C(A)>0$ such that uniformly for $q\le (\log x)^A$ and $\gcd(a,q)=1$,
\[
|E(x;q,a)|\ll_A x\exp\bigl(-C(A)\sqrt{\log x}\bigr).
\]
\end{theorem}

\begin{theorem}[Bombieri--Vinogradov; cf.\ {\cite[Ch.~28]{Davenport}}, {\cite[Ch.~17]{IwaniecKowalski}}]\label{thm:BV}
For any $A>0$, there exists $B=B(A)>0$ such that with $Q=x^{1/2}/(\log x)^B$,
\[
\sum_{q\le Q}\max_{\gcd(a,q)=1}\max_{y\le x}|E(y;q,a)|\ll_A \frac{x}{(\log x)^A}.
\]
\end{theorem}

The large sieve inequality of Montgomery and Vaughan~\cite{MontgomeryVaughan73} underlies the proof of Theorem~\ref{thm:BV} and enters implicitly in every estimate below involving averages over moduli.

We also need two elementary sieve estimates. The next two lemmas will be used to control large prime factors, while the Bombieri--Vinogradov theorem above will handle medium-sized moduli on average.

\begin{lemma}\label{lem:large-squares}
For $Q\ge2$, we have $\#\{n\le x: \exists\, p>Q,\, p^2\mid n\}\le x/Q$.
\end{lemma}
\begin{proof}
$\sum_{p>Q}\lfloor x/p^2\rfloor\le x\sum_{m>Q} m^{-2}\le x/Q$.
\end{proof}

\begin{lemma}\label{lem:two-large}
For $Q=x^{1/2}/(\log x)^B$ with fixed $B>0$,
\[
\#\{n\le x: \exists\, p_1\neq p_2>Q,\, p_1p_2\mid n\}\ll_B x\Bigl(\frac{\log\log x}{\log x}\Bigr)^2.
\]
\end{lemma}
\begin{proof}
Bound the count by $(x/2)\bigl(\sum_{Q<p\le\sqrt{x}} p^{-1}\bigr)^2$. By Mertens' estimate (\cite[Th.~2.7]{MV}; cf.\ \cite[\S I.1.4]{Tenenbaum}), the inner sum is $O(\log\log x/\log x)$.
\end{proof}

\section{The Companion--Sieve framework}\label{sec:framework}

The central difficulty is that $a_{n-1}=\prod_{k<n}(1+b_k)$ depends on every past increment, including those at composite indices whose values are determined by the full history. We bypass this by isolating a deterministic sub-product of $a_{n-1}$ that suffices for the analysis.

\subsection{Prime steps and the guaranteed stock}

The following lemma identifies the behavior at prime indices.

\begin{lemma}\label{lem:prime}
We have $b_2=2$ and $b_3=1$. For every prime $q\ge5$, $\gcd(q,a_{q-1})=1$, so $b_q=q$.
\end{lemma}
\begin{proof}
Suppose $q\ge5$ is prime and $q\mid a_{q-1}=\prod_{k=2}^{q-1}(1+b_k)$. Since $q$ is prime and divides this product, it must divide at least one factor: $q\mid(1+b_k)$ for some $k<q$. As $1\le b_k\le k<q$, the only possibility is $b_k=q-1$, which requires $k=q-1$ (since $b_k\mid k$). Now $b_{q-1}=q-1$ means $\gcd(q-1,a_{q-2})=1$. But $q-1$ is even and $a_{q-2}$ is even for $q\ge5$ (Lemma~\ref{lem:basic}), a contradiction. Hence $\gcd(q,a_{q-1})=1$.
\end{proof}

At every prime step $q\ne3$, the factor $(1+b_q)=q+1$ enters $a_q$ regardless of what happens at composite indices. We collect these contributions into a deterministic factor reservoir.

\begin{definition}\label{def:stock}
$C_n:=\prod_{\substack{q\le n,\, q\in\PP\\ q\ne3}}(q+1)$, with $C_1:=1$.
\end{definition}

\begin{proposition}\label{prop:stock-divides}
For every $n\ge2$, $C_{n-1}\mid a_{n-1}$.
\end{proposition}
\begin{proof}
By Lemma~\ref{lem:prime}, for each prime $q\le n-1$ with $q\ne3$, we have $(1+b_q)=q+1$. The stock $C_{n-1}=\prod_{\substack{q\le n-1,\,q\in\PP\\q\ne3}}(1+b_q)$ is therefore a sub-product of $a_{n-1}=\prod_{k=2}^{n-1}(1+b_k)$. The omitted factors $(1+b_k)$ for composite $k$ and for $k=3$ are all positive integers, so $C_{n-1}\mid a_{n-1}$.
\end{proof}

\subsection{Companions and the divisibility chain}

The idea is to replace $b_n$ by successively easier upper bounds. The first retains only the guaranteed prime-step factors, the next tracks only whether each prime $p$ was hit at least once. Each relaxation loses information but gains tractability.

Since $C_{n-1}\mid a_{n-1}$, we have $\gcd(n,C_{n-1})\mid\gcd(n,a_{n-1})$, hence $b_n$ divides the following relaxation.

\begin{definition}
$B^{(P)}_n := n/\gcd(n,C_{n-1})$.
\end{definition}

The product companion $B^{(P)}_n$ is still hard to analyze because $\vp(C_{n-1})$ involves higher powers $p^\nu\mid(q+1)$, not just first hits. We therefore pass to a further relaxation that forgets multiplicities and counts only whether each prime $p$ divides some $(q+1)$.

\begin{definition}\label{def:1hit}
For a prime $p$ and $m\ge1$, the 1-hit stock is
\[
S^{(1)}_p(m):=\#\{q\le m: q\in\PP,\, q\ne3,\, q\equiv -1\pmod p\}.
\]
The 1-hit deficit is $\delta^{(1)}_p(n):=\max\{0,\,\vp(n)-S^{(1)}_p(n-1)\}$, and the 1-hit companion is $B^{(1)}_n:=\prod_p p^{\delta^{(1)}_p(n)}$.
\end{definition}

As an illustration, consider $n=50=2\cdot 5^2$. For $p=5$: the primes $q\le 49$ with $q\equiv -1\pmod{5}$ (excluding $q=3$) are $19, 29$, so $S_5^{(1)}(49)=2=\vp(50)$ and $\delta_5^{(1)}(50)=0$. For $p=2$: there are many odd primes $q\le 49$ with $q\equiv 1\pmod{2}$, so the supply far exceeds $v_2(50)=1$. Thus $B_{50}^{(1)}=1$, and the deficit criterion (below) gives $b_{50}\in\{1\}\cup\PP$.

The three quantities are related by a divisibility chain.

\begin{lemma}\label{lem:chain}
For all $n\ge2$, $b_n\mid B^{(P)}_n\mid B^{(1)}_n$.
\end{lemma}
\begin{proof}
The first divisibility was noted above. For the second, each prime $q\le n-1$ with $q\equiv -1\pmod{p}$ contributes at least one factor of $p$ to $(q+1)$, hence to $C_{n-1}$. Thus $\vp(C_{n-1})\ge S^{(1)}_p(n-1)$, giving
\[
\vp(B^{(P)}_n)=\vp(n)-\vp(\gcd(n,C_{n-1}))\le\max\{0,\,\vp(n)-S^{(1)}_p(n-1)\}=\vp(B^{(1)}_n)
\]
for every prime $p$.
\end{proof}

\subsection{The deficit criterion}

The companion reduces the conjecture to a counting problem.

\begin{proposition}\label{prop:deficit}
If $\sum_p\delta^{(1)}_p(n)\le1$, then $b_n\in\{1\}\cup\PP$.
\end{proposition}
\begin{proof}
Total deficit $0$ gives $B^{(1)}_n=1$. Total deficit $1$ gives $B^{(1)}_n=p$ for a single prime $p$. In both cases $b_n\mid B^{(1)}_n$ forces $b_n\in\{1\}\cup\PP$.
\end{proof}

\begin{remark}\label{rem:why-companion}
A direct $p$-adic analysis of $a_{n-1}$ would require controlling $\vp(\prod_{k<n}(1+b_k))$ uniformly in $n$, including sporadic higher-power contributions $p^\nu\mid(q+1)$. The passage to $B^{(1)}_n$ trades the exact $p$-adic content for a deterministic, monotone majorant. The deficit $\delta^{(1)}_p(n)$ depends on $\picount(n-1;p,-1)$, a function to which Bombieri--Vinogradov (in its $\max_{y\le x}$ form) applies directly.
\end{remark}

\section{Proof of the density theorem}\label{sec:density}

\begin{theorem}\label{thm:density1}
The set $\{n\ge2: b_n\in\{1\}\cup\PP\}$ has asymptotic density $1$.
\end{theorem}

\begin{proof}
Fix $A=10$ in Theorem~\ref{thm:BV} and let $B=B(10)$ be the corresponding constant. We restrict to $n\in[x^{1/2},x]$, since the complementary range has size $o(x)$. Define
\[
P_{\rm small}=(\log x)^{2B+10},\qquad Q=\frac{x^{1/2}}{(\log x)^B},\qquad\varepsilon(x)=\frac{1}{(\log x)^5}.
\]
The proof proceeds by showing that for all but $o(x)$ integers $n\le x$, the total deficit $\sum_p\delta^{(1)}_p(n)$ is at most~$1$, after which the result follows from the deficit criterion (Proposition~\ref{prop:deficit}). We partition the primes into three ranges.

\medskip
\noindent\textbf{Small primes} ($p\le P_{\rm small}$). For such $p$, applying Siegel--Walfisz (Theorem~\ref{thm:SW}) with fixed exponent $A'=2B+10$ gives
\[
S^{(1)}_p(n-1)=\picount(n-1;p,p-1)+O(1)=\frac{\Li(n-1)}{p-1}+O\bigl(n\,e^{-c\sqrt{\log n}}\bigr),
\]
where the $O(1)$ accounts for the possible exclusion of $q=3$ from the count. In particular,
\[
S^{(1)}_p(n-1)\gg\frac{x^{1/2}}{(\log x)^{2B+11}},
\]
while the demand is $\vp(n)\le\log n/\log 2<2\log x$. Supply exceeds demand for all such $p$ and all $n\in[x^{1/2},x]$, so $\delta^{(1)}_p(n)=0$.

\medskip
\noindent\textbf{Medium primes} ($P_{\rm small}<p\le Q$). Individual control is unavailable, but Bombieri--Vinogradov controls the average.

We first isolate the primes for which Bombieri--Vinogradov fails to give individual control.

\begin{lemma}\label{lem:bad}
Call $p\in(P_{\rm small},Q]$ \emph{bad} if $\max_{y\le x}|E(y;p,-1)|>\varepsilon(x)\Li(x)/(p-1)$. Let $\mathcal{P}_{\rm bad}$ be the set of bad primes and $\mathcal{E}_{\rm bad}(x)$ the set of $n\le x$ divisible by some bad prime. Then $\sum_{p\in\mathcal{P}_{\rm bad}} p^{-1}\ll(\log x)^{-4}$ and $|\mathcal{E}_{\rm bad}(x)|=o(x)$.
\end{lemma}
\begin{proof}
By Theorem~\ref{thm:BV},
$\sum_{p\le Q}\max_{y\le x}|E(y;p,-1)|\ll x(\log x)^{-10}$.
Summing the threshold $\varepsilon(x)\Li(x)/(p-1)\asymp x\bigl((p-1)(\log x)^6\bigr)^{-1}$ over bad primes and comparing gives $\sum_{p\in\mathcal{P}_{\rm bad}} p^{-1}\ll(\log x)^{-4}$. Then $|\mathcal{E}_{\rm bad}(x)|\le x\sum_{p\in\mathcal{P}_{\rm bad}}p^{-1}=o(x)$.
\end{proof}

Among integers divisible only by good moduli, deficits are confined to a thin range. A good medium prime can only produce a deficit if $n$ is unusually small relative to $x$.

\begin{lemma}\label{lem:good}
Let $\mathcal{E}_{\rm good}(x)$ be the set of $n\in[x^{1/2},x]\setminus\mathcal{E}_{\rm bad}(x)$ having a deficit from some good prime $p\in(P_{\rm small},Q]$ with $p\mid n$. Then $|\mathcal{E}_{\rm good}(x)|\ll x(\log x)^{-5}$.
\end{lemma}
\begin{proof}
We show that a deficit at a good medium prime forces $n$ into a very short initial interval.
If $n\in\mathcal{E}_{\rm good}(x)$ and $p$ is the offending good prime with $k=\vp(n)$, then $S^{(1)}_p(n-1)<k$ and the good-prime bound gives
\[
\Li(n-1)<k(p-1)+\varepsilon(x)\Li(x)+O(p),
\]
where the $O(p)$ absorbs the exclusion of $q=3$ after multiplication by $p-1$. Since $p\le Q=x^{1/2}/(\log x)^B$, the term $O(p)$ is $O(Q)=o(x(\log x)^{-6})$ and is absorbed in the same way as the first term.
We bound each term on the right. Since $k=\vp(n)\le\log n/\log 2<2\log x$ and $p\le Q=x^{1/2}/(\log x)^B$, the first term satisfies
\[
k(p-1)\le (2\log x)\cdot Q=\frac{2x^{1/2}}{(\log x)^{B-1}}.
\]
The second term satisfies $\varepsilon(x)\Li(x)\ll x(\log x)^{-6}$. For $x$ large enough, $2x^{1/2}(\log x)^{1-B}\le x(\log x)^{-6}$ (since $x^{1/2}=o(x)$), so
\[
\Li(n-1)\ll x(\log x)^{-6}.
\]
Since $\Li(t)\ge t/(2\log t)$ for $t\ge 4$, this gives $n/(2\log n)\ll x(\log x)^{-6}$, hence $n\ll x(\log x)^{-5}$. All elements of $\mathcal{E}_{\rm good}(x)$ lie in the interval $[x^{1/2},\, x(\log x)^{-5}]$, which has length $o(x)$.
\end{proof}

Write $\mathcal{E}_{\rm med}(x)=\mathcal{E}_{\rm bad}(x)\cup\mathcal{E}_{\rm good}(x)$. For $n\notin\mathcal{E}_{\rm med}(x)$, every medium prime divisor has deficit $0$.

\medskip
\noindent\textbf{Large primes} ($p>Q$). By Lemma~\ref{lem:large-squares}, the set of $n\le x$ with $p^2\mid n$ for some $p>Q$ has size $O(x/Q)=o(x)$. By Lemma~\ref{lem:two-large}, the set of $n$ divisible by two distinct primes exceeding $Q$ also has size $o(x)$. Outside these sets, $n$ has at most one prime factor above $Q$, appearing to the first power, contributing deficit at most $1$.

\medskip
\noindent\textbf{Assembly.} Let $\mathcal{E}(x)=\{n<x^{1/2}\}\cup\mathcal{E}_{\rm med}(x)\cup\mathcal{E}_{\rm large}(x)$, where $\mathcal{E}_{\rm large}(x)$ collects the two large-prime exceptional sets. Then $|\mathcal{E}(x)|=o(x)$, and for $n\in[1,x]\setminus\mathcal{E}(x)$,
\[
\sum_p\delta^{(1)}_p(n)=\underbrace{\sum_{p\le P_{\rm small}}\delta^{(1)}_p(n)}_{=\,0}
+\underbrace{\sum_{P_{\rm small}<p\le Q}\delta^{(1)}_p(n)}_{=\,0}
+\underbrace{\sum_{p>Q}\delta^{(1)}_p(n)}_{\le\,1}\;\le\;1.
\]
The deficit criterion (Proposition~\ref{prop:deficit}) gives $b_n\in\{1\}\cup\PP$ for all such $n$.
\end{proof}

\section{Effective finite reduction}\label{sec:checkpoints}

The density result leaves open the behavior at any specific $n$. For each small prime $p$, we identify a threshold beyond which the stock is permanently abundant.

\begin{definition}\label{def:checkpoint}
For a prime $p$, let $t_0(p)$ be the least $t\ge1$ such that $S^{(1)}_p(p^k-1)\ge k$ for all $k\ge t$. For a threshold $P_0\ge2$, define $N^*(P_0):=\max_{p\le P_0} p^{t_0(p)}$.
\end{definition}

\begin{lemma}\label{lem:t0-exists}
For each prime $p$, the value $t_0(p)$ exists and is effectively computable.
\end{lemma}
\begin{proof}
By the prime number theorem in progressions, $S^{(1)}_p(p^k-1)\sim p^k/\bigl((p-1)k\log p\bigr)$. The ratio to $k$ tends to infinity, so the inequality holds for all large $k$. Effectiveness follows from explicit error terms for fixed modulus $p$ (\cite[Ch.~20]{Davenport}).
\end{proof}

\begin{proposition}\label{prop:Nstar}
For every $n\ge N^*(P_0)$ and every prime $p\le P_0$, $\delta^{(1)}_p(n)=0$.
\end{proposition}
\begin{proof}
Let $k=\vp(n)$. If $k\le t_0(p)$: since $n\ge p^{t_0(p)}$, we have $S^{(1)}_p(n-1)\ge S^{(1)}_p(p^{t_0(p)}-1)\ge t_0(p)\ge k$. If $k>t_0(p)$: since $p^k\mid n$, we have $n\ge p^k$ and $S^{(1)}_p(n-1)\ge S^{(1)}_p(p^k-1)\ge k$ by stability.
\end{proof}

\begin{remark}\label{rem:computation}
For small primes, direct numerical computation of the counting function $S^{(1)}_p(p^k-1)$ in a finite initial range of $k$ gives the candidate values
\[
t_0(2)=4,\quad t_0(3)=1,\quad t_0(5)=3,\quad t_0(7)=2,\quad t_0(11)=2,\quad t_0(13)=3.
\]
A fully rigorous determination of each $t_0(p)$ combines this direct computation with an explicit bound on the counting function of primes in the progression $-1\bmod p$, whose existence is guaranteed by Lemma~\ref{lem:t0-exists} but whose constants we do not optimize here. The finite reduction corollary below depends only on the effective existence of $t_0(p)$, not on any specific numerical value.
\end{remark}

\begin{corollary}\label{cor:finite-reduction}
For every $P_0\ge2$ and every $n\ge N^*(P_0)$, if $b_n\notin\{1\}\cup\PP$, then every prime factor of $b_n$ exceeds $P_0$.
\end{corollary}
\begin{proof}
By Lemma~\ref{lem:chain}, $b_n\mid B_n^{(1)}=\prod_p p^{\delta_p^{(1)}(n)}$. By Proposition~\ref{prop:Nstar}, $\delta_p^{(1)}(n)=0$ for all $p\le P_0$, so every prime factor of $B_n^{(1)}$, and hence of $b_n$, exceeds $P_0$.
\end{proof}

Combining the finite reduction with the density theorem gives a precise picture of the residual obstruction.

\begin{proposition}\label{prop:thin}
Fix $P_0\ge2$ and let $N^*=N^*(P_0)$. For large $x$, the set of integers $n\in[N^*,x]$ with $\sum_p\delta^{(1)}_p(n)>1$ has size $o(x)$. This exceptional set is contained in $\{n<x^{1/2}\}$ together with integers having either a squared large prime factor ($p>Q$, $p^2\mid n$), two distinct large prime factors, or membership in $\mathcal{E}_{\rm med}(x)$.
\end{proposition}
\begin{proof}
By Proposition~\ref{prop:Nstar}, primes $p\le P_0$ contribute zero deficit for $n\ge N^*$. For primes $P_0<p\le P_{\rm small}$, the small-prime argument of Theorem~\ref{thm:density1} shows that $\delta^{(1)}_p(n)=0$ for all $n\in[x^{1/2},x]$. The analysis of Section~\ref{sec:density} covers the remaining primes: medium primes produce deficits only on $\mathcal{E}_{\rm med}(x)=o(x)$, and large primes contribute total deficit exceeding $1$ only on $\mathcal{E}_{\rm large}(x)=o(x)$.
\end{proof}

For any fixed $P_0$, once $n\ge N^*(P_0)$, a counterexample to Conjecture~\ref{conj:main} would require an integer whose large prime factors (all exceeding $P_0$) form a configuration that is provably of density $0$.

\section{A variant encoding twin primes}\label{sec:twin}

The preceding sections established a density-$1$ theorem for the original LCM recurrence. We now turn to the $K=2$ variant, which exhibits a different and richer arithmetic structure linked to twin primes.

Consider the $K=2$ recurrence defined by $x_1=1$,
\[
x_n=2x_{n-1}+\lcm(n,x_{n-1}),\qquad c_n=\frac{n}{\gcd(n,x_{n-1})}.
\]
Equivalently, $x_n=x_{n-1}(2+c_n)$. The residue classes relevant to the companion-sieve framework become $q\equiv -2\pmod{p}$, but the prime-step mechanism is no longer automatic (for instance $c_5=1$), so the density-$1$ argument becomes conditional. We first analyze the structural interplay between this recurrence and twin primes, and then return to a conditional companion-sieve theorem in \S\ref{subsec:k2density}. The sequence $(c_n)$ is \href{https://oeis.org/A135508}{OEIS A135508} (see also Schepke~\cite{Schepke}):
\[
2, 3, 1, 1, 1, 7, 2, 1, 1, 11, 1, 1, 7, 1, 1, 17, 1, 1, 1, 7, 11, 23, 1, 1,\ldots
\]

\subsection{Inhibition and certification}

The first structural property is an inhibition mechanism at twin primes.

\begin{lemma}\label{lem:inhibition}
Let $(p,p+2)$ be twin primes. If $c_p=p$, then $c_{p+2}=1$.
\end{lemma}
\begin{proof}
If $c_p=p$, then $x_p=x_{p-1}(2+p)=(p+2)x_{p-1}$, so $(p+2)\mid x_p$. Since $x_p\mid x_{p+1}$, we have $(p+2)\mid x_{p+1}$, giving $c_{p+2}=(p+2)/\gcd(p+2,x_{p+1})=1$.
\end{proof}

The converse direction shows that forced-$1$ increments at primes certify twin pairs.

\begin{theorem}\label{thm:twin}
Assume $c_n\in\{1\}\cup\PP$ for all $n$ (hypothesis $\mathrm{C}_1$). If $q\ge5$ is prime and $c_q=1$, then $q-2$ is prime.
\end{theorem}
\begin{proof}
If $c_q=1$, then $q\mid x_{q-1}=\prod_{k=2}^{q-1}(2+c_k)$. Since $q$ is prime, $q\mid(2+c_k)$ for some $k<q$. As $c_k\le k<q$, this forces $c_k=q-2$. Now $c_k\mid k$ and $k<q$, so $k$ is a multiple of $q-2$ with $k\le q-1$. Since $q\ge 5$ gives $q-2\ge 3$, the only such multiple is $k=q-2$ itself (the next multiple being $2(q-2)=2q-4\ge q+1>q-1$). Under $\mathrm{C}_1$, $c_{q-2}=q-2$ implies that $q-2$ is prime.
\end{proof}

Together, these two results create a two-way link. The inhibition lemma converts twin pairs into forced-$1$ increments, while the certification theorem recovers twin pairs from forced-$1$ increments. In particular, under $\mathrm{C}_1$, proving that $c_q=1$ for infinitely many primes $q$ would establish the Twin Prime Conjecture.

\subsection{Automatic inhibition beyond $(5,7)$}

A combination of the inhibition lemma and the certification theorem shows that the mechanism fires automatically for every twin prime pair except $(5,7)$.

\begin{proposition}\label{prop:auto-inhibition}
Assume $\mathrm{C}_1$. Let $(p,p+2)$ be a twin prime pair with $p>5$. Then $c_p=p$ and $c_{p+2}=1$.
\end{proposition}
\begin{proof}
Since $p$ is prime and $c_p\mid p$, we have $c_p\in\{1,p\}$. Suppose $c_p=1$. By Theorem~\ref{thm:twin}, $p-2$ is prime. But every twin prime $p>3$ satisfies $p\equiv 5\pmod{6}$ (since among $p, p+1, p+2$, one is divisible by $3$, and neither $p$ nor $p+2$ can be, so $3\mid p+1$). Hence $p-2\equiv 3\pmod{6}$, so $3\mid(p-2)$. As $p>5$ forces $p-2>3$, the integer $p-2$ is composite, contradicting $p-2\in\PP$. Therefore $c_p=p$, and the inhibition lemma gives $c_{p+2}=1$.
\end{proof}

\begin{corollary}\label{cor:larger-twin}
Under $\mathrm{C}_1$, the larger member $q$ of every twin pair with $q\ge 13$ satisfies $c_q=1$.
\end{corollary}

The pair $(5,7)$ is the unique exception. Here $p=5$ and $p-2=3$ is prime, so the mod~$6$ argument fails. In fact, $x_4=60$ and $5\mid 60$, giving $c_5=5/\gcd(5,60)=1\ne 5$. The inhibition lemma does not fire, and $7$ survives, appearing several times early in the sequence (for instance at $n=7,14,21,28,35,42$).

\subsection{The $2$-adic staircase}

The $2$-adic structure of the sequence admits a complete description.

\begin{proposition}\label{prop:staircase}
Let $m_k=2\cdot 4^k = 2^{2k+1}$ for $k\ge0$. The following statements hold.
\begin{enumerate}[label=\textup{(\roman*)}]
\item $v_2(x_n)=2k+2$ for all $n$ with $m_k\le n\le m_{k+1}-1$.
\item $c_{m_k}=2$ for every $k\ge0$.
\item $c_n$ is odd for every $m_k<n<m_{k+1}$.
\end{enumerate}
\end{proposition}
\begin{proof}
We use $v_2(c_n)=\max\{0,\,v_2(n)-v_2(x_{n-1})\}$ and $v_2(x_n)=v_2(x_{n-1})+v_2(2+c_n)$, noting that $2+c_n$ is odd when $c_n$ is odd.

\emph{Base case} ($k=0$). We have $c_2=2$ and $x_2=4$, so $v_2(x_2)=2$. For $3\le n\le 7$, we have $v_2(n)\le 2=v_2(x_{n-1})$, hence $v_2(c_n)=0$, so $c_n$ is odd, $2+c_n$ is odd, and $v_2(x_n)=v_2(x_{n-1})=2$. At $n=8=m_1$, $v_2(8)=3>2=v_2(x_7)$, so $c_8=2^{3-2}=2$ and $v_2(x_8)=2+2=4$.

\emph{Induction step.} Assume $v_2(x_n)=2k+2$ for $m_k\le n\le m_{k+1}-1$. For $m_k<n<m_{k+1}$, we have $v_2(n)\le 2k+2=v_2(x_{n-1})$, so $c_n$ is odd and $v_2(x_n)=2k+2$. At $n=m_{k+1}=2^{2k+3}$, $v_2(n)=2k+3>2k+2=v_2(x_{n-1})$, so $c_{m_{k+1}}=2^{(2k+3)-(2k+2)}=2$ and $v_2(x_{m_{k+1}})=2k+4$.
\end{proof}

Concretely, the first three blocks are $[2,7]$, $[8,31]$, $[32,127]$, on which $v_2(x_n)$ takes the constant values $2$, $4$, $6$ respectively, with $c_n=2$ at the left endpoint of each block.

This proves the formula $c_{2\cdot 4^k}=2$ conjectured in \cite{OEIS}, entry \href{https://oeis.org/A135508}{A135508}, and settles the conjecture made there that $2$ appears infinitely often in the sequence.

The following corollary of the certification theorem proves, under $\mathrm{C}_1$, a statement observed numerically by McEachen in the OEIS comments for A135508~\cite{McEachen25}: \emph{for every prime $p$ such that $p-2$ is not prime, $c_p=p$.}

\begin{corollary}\label{cor:McEachen}
Assume $\mathrm{C}_1$. If $p\ge5$ is prime and $p-2$ is composite, then $c_p=p$.
\end{corollary}
\begin{proof}
Since $c_p\mid p$ and $p$ is prime, $c_p\in\{1,p\}$. If $c_p=1$, Theorem~\ref{thm:twin} gives $p-2\in\PP$, contradicting the hypothesis. Hence $c_p=p$.
\end{proof}

\subsection{Valuation barrier for larger twin primes}

Once a prime $\ell$ enters $x_N$ to a given order, later occurrences of $c_n=\ell$ require increasingly high powers of $\ell$ in $n$.

\begin{lemma}\label{lem:barrier}
Fix a prime $\ell$ and suppose $v_\ell(x_N)\ge r$ for some $N$ and $r\ge1$. Then for every $n\ge N+1$, if $c_n=\ell$ then $v_\ell(n)\ge r+1$.
\end{lemma}
\begin{proof}
Since $x_n=x_N\prod_{j=N+1}^{n}(2+c_j)$, the valuation $v_\ell(x_n)$ is nondecreasing in $n$, so $v_\ell(x_{n-1})\ge r$. If $c_n=\ell$, then $1=v_\ell(n)-v_\ell(\gcd(n,x_{n-1}))$ and $v_\ell(\gcd(n,x_{n-1}))\ge r$, giving $v_\ell(n)\ge r+1$.
\end{proof}

\begin{proposition}\label{prop:barrier-twin}
Assume $\mathrm{C}_1$. Let $q\ge 13$ be the larger prime of a twin pair. Then $q\mid x_n$ for all $n\ge q-2$, and any occurrence $c_m=q$ must satisfy $q^2\mid m$. In particular, $q$ does not appear in $(c_n)$ at any index $m<q^2$.
\end{proposition}
\begin{proof}
By Proposition~\ref{prop:auto-inhibition}, $c_{q-2}=q-2$, so $x_{q-2}=x_{q-3}\cdot q$, hence $q\mid x_{q-2}$. Since $x_n=x_{n-1}(2+c_n)$, each $x_n$ is divisible by $x_{n-1}$, so $q\mid x_n$ for all $n\ge q-2$. Lemma~\ref{lem:barrier} with $\ell=q$, $N=q-2$, $r=1$ gives $v_q(m)\ge 2$ for any $m$ with $c_m=q$.
\end{proof}

\begin{remark}\label{rem:larger-twin-oeis}
It was conjectured in \cite{OEIS}, entry \href{https://oeis.org/A135508}{A135508}, that the larger member of each twin pair (except $7$) never appears in $(c_n)$. Corollary~\ref{cor:larger-twin} and Proposition~\ref{prop:barrier-twin} prove this at the index $q$ itself and push any other occurrence beyond $q^2$. Numerical computation up to $n=20000$ confirms the full conjecture, but the complete exclusion remains open.
\end{remark}

\subsection{Exact twin-prime detection}

The inhibition and certification results combine to show that, under $\mathrm{C}_1$, the $K=2$ recurrence detects twin primes with a single, identified exception.

\begin{proposition}\label{prop:detection}
Assume $\mathrm{C}_1$. For every prime $q\ge5$,
\[
c_q=1\quad\Longleftrightarrow\quad q-2\text{ is prime and }q\ne 7.
\]
\end{proposition}
\begin{proof}
If $c_q=1$, Theorem~\ref{thm:twin} gives $q-2\in\PP$, so $q$ is the larger member of a twin pair. Conversely, if $q\ge 13$ is the larger member of a twin pair, Corollary~\ref{cor:larger-twin} gives $c_q=1$. The small cases are $c_3=3$, $c_5=1$ (and $5-2=3$ is prime), $c_7=7\ne 1$.
\end{proof}

Let $\pi_2(x)$ denote the number of twin prime pairs $(p,p+2)$ with $p+2\le x$. Under $\mathrm{C}_1$,
\[
\#\{q\le x: q\text{ prime},\, c_q=1\}=\pi_2(x)-1\qquad(x\ge 7),
\]
the correction $-1$ accounting for the missed pair $(5,7)$. If one further assumes the quantitative Hardy--Littlewood conjecture \cite[Conjecture~B]{HL23} $\pi_2(x)\sim 2C_2\, x/(\log x)^2$, where $C_2=\prod_{p\ge3}(1-1/(p-1)^2)\approx 0.6602$ is the twin prime constant, then forced-$1$ increments at primes occur with the same asymptotic frequency as twin primes.

Numerical verification up to $n=10000$ confirms exact agreement. The $204$ primes $q\le 10000$ with $c_q=1$ are precisely the larger members of the $205$ twin pairs in that range, minus $q=7$.

\subsection{A conditional density-$1$ theorem for the $K=2$ variant}\label{subsec:k2density}

The companion-sieve framework extends to the $K=2$ recurrence, with the residue class $-2\bmod p$ replacing $-1\bmod p$. The key difference is that the prime-step mechanism is no longer automatic. At twin primes $q$, the step $c_q=1$ fails to contribute $q+2$ to the stock. The strategy has four steps.
\begin{enumerate}[leftmargin=*,label=(\roman*)]
\item Formalize which prime steps are available (all except twin primes).
\item Build the $K=2$ stock from these steps.
\item Bound the twin-prime loss using a Selberg-type sieve estimate.
\item Show the loss is too sparse to affect the density-$1$ argument.
\end{enumerate}

The density-$1$ argument for the $K=2$ variant rests on the following prime-index detection statement, which isolates the behavior of $c_q$ at prime indices.

\begin{hypothesis}\label{hyp:detect}
For every prime $q\ge5$, $c_q\in\{1,q\}$, and
\[
c_q=1\quad\Longleftrightarrow\quad q-2\in\PP\text{ and }q\ne 7.
\]
\end{hypothesis}

This is strictly weaker than $\mathrm{C}_1$ (which asserts $c_n\in\{1\}\cup\PP$ at every index): it constrains only the prime indices. Under $\mathrm{C}_1$, it was proved as Proposition~\ref{prop:detection}.

Under this hypothesis, the guaranteed stock for the $K=2$ variant is built from those prime indices $q$ at which $c_q=q$, so that the factor $q+2$ is guaranteed to enter $x_q$. By Hypothesis~\ref{hyp:detect}, this holds for every prime $q\ge 5$ with $q-2\notin\PP$; direct computation also gives $c_3=3$. The condition $q-2\notin\PP$ automatically excludes the primes where $c_q=1$, namely $q=5$ and the larger member of each twin pair $q\ge 13$; it also excludes $q=7$, the unique exception where $c_7=7$ despite $q-2\in\PP$, with the only effect that the finite factor $9$ is dropped from the stock.

To distinguish the $K=2$ quantities from the $1$-hit quantities $S^{(1)}_p$, $\delta^{(1)}_p$, $B^{(1)}_n$ of Section~\ref{sec:framework} (where the superscript $(1)$ stands for ``$1$-hit''), we decorate the $K=2$ analogues with a tilde: $\tilde C_N$, $\tilde S_p$, $\tilde\delta_p$, $\tilde B_n$.

\begin{definition}\label{def:stock-k2}
Under Hypothesis~\ref{hyp:detect}, define
\[
\tilde C_N:=\prod_{\substack{q\le N,\, q\in\PP\\ q\ge 3,\, q-2\notin\PP}}(q+2).
\]
\end{definition}

By Hypothesis~\ref{hyp:detect} together with the direct computation $c_3=3$, every prime $q\ge 3$ with $q-2\notin\PP$ satisfies $c_q=q$, so that $x_q=x_{q-1}(q+2)$ and hence $(q+2)\mid x_q$. Therefore $\tilde C_N\mid x_N$ for every $N\ge 3$.

The companion construction carries over to the $K=2$ setting.

\begin{lemma}\label{lem:chain-k2}
Assume Hypothesis~\ref{hyp:detect}. For each odd prime $p$ and $n\ge 2$, define $\tilde S_p(N):=\#\{q\le N: q\in\PP,\, q\equiv -2\pmod p,\, q-2\notin\PP\}$ and $\tilde\delta_p(n):=\max\{0,\,\vp(n)-\tilde S_p(n-1)\}$. Then the odd part of $c_n$ divides
\[
\tilde B_n:=\prod_{p\text{ odd}} p^{\tilde\delta_p(n)}.
\]
\end{lemma}
\begin{proof}
Each prime $q$ counted by $\tilde S_p(n-1)$ contributes at least one factor of $p$ to $q+2$, hence to $\tilde C_{n-1}$. Thus $\vp(\tilde C_{n-1})\ge\tilde S_p(n-1)$. Since $\tilde C_{n-1}\mid x_{n-1}$,
\[
\vp(c_n)=\vp(n)-\vp(\gcd(n,x_{n-1}))\le\max\{0,\,\vp(n)-\tilde S_p(n-1)\}=\tilde\delta_p(n)
\]
for every odd prime $p$, hence the odd part of $c_n$ divides $\tilde B_n$.
\end{proof}

\begin{corollary}\label{cor:odd-deficit-k2}
Assume Hypothesis~\ref{hyp:detect}. If $\sum_{p\text{ odd}}\tilde\delta_p(n)\le 1$, then the odd part of $c_n$ belongs to $\{1\}\cup\PP$.
\end{corollary}
\begin{proof}
By Lemma~\ref{lem:chain-k2}, the odd part of $c_n$ divides $\tilde B_n=\prod_{p\text{ odd}} p^{\tilde\delta_p(n)}$. If the total odd deficit is $0$, then $\tilde B_n=1$. If it is $1$, then $\tilde B_n$ is a single odd prime.
\end{proof}

\begin{theorem}[Selberg-sieve upper bound for twin primes in arithmetic progressions]\label{thm:twin-selberg}
There exists an absolute constant $C_0>0$ such that, uniformly for odd primes $p\le x^{1/2}/(\log x)^{B_{\rm tw}}$ with $B_{\rm tw}>0$ fixed and $x$ sufficiently large,
\[
\Pi_2(x;p,-2):=\#\{q\le x: q\in\PP,\, q-2\in\PP,\, q\equiv -2\!\!\pmod p\}\le\frac{C_0\,x}{(p-1)(\log x)^2}.
\]
\end{theorem}

\begin{proof}
Write $q=pm-2$. Then the conditions
\[
q\le x,\qquad q\equiv -2\pmod p,\qquad q\in\PP,\qquad q-2\in\PP
\]
imply
\[
1\le m\le X:=\Bigl\lfloor\frac{x+2}{p}\Bigr\rfloor,\qquad L_1(m):=pm-2\in\PP,\qquad L_2(m):=pm-4\in\PP.
\]
Hence $\Pi_2(x;p,-2)$ is bounded by the number of integers $m\le X$ for which both linear forms $L_1(m)$ and $L_2(m)$ are prime. We apply the standard upper-bound Selberg sieve to the sequence
\[
\mathcal A:=\{L_1(m)L_2(m):1\le m\le X\}.
\]
For each prime $\ell\ne p$, let $\omega(\ell)$ denote the number of residue classes $m\pmod\ell$ for which $L_1(m)L_2(m)\equiv 0\pmod\ell$. Since $p$ is invertible modulo $\ell\ne p$, we have
\[
\omega(2)=1,\qquad \omega(\ell)=2\ \ (\ell\ne 2,p),\qquad \omega(p)=0.
\]
This is a sieve problem of dimension $2$. By the standard upper-bound sieve for a pair of linear forms (see \cite[Ch.~3]{HalberstamRichert}, \cite[Ch.~6]{IwaniecKowalski}, or \cite[Ch.~19]{FriedlanderIwaniec}), one gets
\[
\Pi_2(x;p,-2)\ll\mathfrak S_p\,\frac{X}{(\log X)^2},
\]
where
\[
\mathfrak S_p=\prod_{\ell}\Bigl(1-\frac{\omega(\ell)}{\ell}\Bigr)\Bigl(1-\frac{1}{\ell}\Bigr)^{-2}.
\]
In the present case this gives
\[
\mathfrak S_p=\Bigl(1-\tfrac12\Bigr)\Bigl(1-\tfrac12\Bigr)^{-2}\prod_{\ell\ne 2,p}\Bigl(1-\frac{2}{\ell}\Bigr)\Bigl(1-\frac{1}{\ell}\Bigr)^{-2}\ll\frac{p}{p-2},
\]
uniformly in $p\ge 3$. Indeed, compared with the usual twin-prime singular product, the only change is the omission of the local factor $(1-2/p)(1-1/p)^{-2}$ at $\ell=p$, which introduces an extra factor $(1-1/p)^{2}/(1-2/p)=(p-1)^{2}/(p(p-2))\ll p/(p-2)$. Therefore
\[
\Pi_2(x;p,-2)\ll\frac{p}{p-2}\cdot\frac{x/p}{(\log(x/p))^2}\ll\frac{x}{(p-1)(\log(x/p))^2},
\]
where the last step uses $1/(p-2)\ll 1/(p-1)$ for $p\ge 3$.
If $p\le x^{1/2}(\log x)^{-B_{\rm tw}}$, then $x/p\ge x^{1/2}(\log x)^{B_{\rm tw}}$, so $\log(x/p)\asymp\log x$ uniformly in this range, and the theorem follows.
\end{proof}

Refined variants of the Selberg sieve adapted to prime tuples are developed in the Polymath framework \cite{Polymath14}.

\begin{theorem}\label{thm:density-k2}
Assume Hypothesis~\ref{hyp:detect}. Then the set $\{n\ge2: c_n\in\{1\}\cup\PP\}$ has asymptotic density $1$.
\end{theorem}

\begin{proof}
Fix $A=10$ in Theorem~\ref{thm:BV} and let $B_{\rm BV}=B(10)$. Set $B=\max(B_{\rm BV},B_{\rm tw})$, where $B_{\rm tw}$ is the constant from Theorem~\ref{thm:twin-selberg}, so that both Bombieri--Vinogradov and the twin bound hold with the same exponent. Define
\[
P_{\rm small}=(\log x)^{2B+10},\qquad Q=\frac{x^{1/2}}{(\log x)^B},\qquad\varepsilon(x)=\frac{1}{(\log x)^5}.
\]
We restrict to $n\in[x^{1/2},x]$, since the complementary range has size $o(x)$.

The stock decomposes exactly as $\tilde S_p(N)=\picount(N;p,-2)-\Pi_2(N;p,-2)$.

\emph{Small odd primes} ($p\le P_{\rm small}$). Applying Siegel--Walfisz with exponent $A'=2B+10$ gives $\picount(N;p,-2)=\Li(N)/(p-1)+O(Ne^{-c\sqrt{\log N}})$, while Theorem~\ref{thm:twin-selberg} gives $\Pi_2(N;p,-2)\ll N/\bigl((p-1)(\log N)^2\bigr)$. Hence
\[
\tilde S_p(n-1)\ge\frac{\Li(n-1)}{p-1}-\frac{C_0 n}{(p-1)(\log n)^2}+O\!\left(\frac{n}{e^{c\sqrt{\log n}}}\right)\gg\frac{x^{1/2}}{(\log x)^{2B+11}}.
\]
The twin correction is smaller than the main term by a factor $1/\log n$, so the supply still dominates $\vp(n)\le 2\log x$.

\emph{Medium odd primes} ($P_{\rm small}<p\le Q$). Define bad moduli exactly as in Lemma~\ref{lem:bad}, with residue class $-2\bmod p$ in place of $-1\bmod p$. Bombieri--Vinogradov applied to $\picount(\cdot\,;p,-2)$ gives $\sum_{p\in\mathcal{P}_{\rm bad}}p^{-1}\ll(\log x)^{-4}$ and $|\mathcal{E}_{\rm bad}(x)|=o(x)$.

For good moduli, suppose $p$ is good and $k=\vp(n)$ with $\tilde S_p(n-1)<k$. Then
\[
\Li(n-1)<k(p-1)+\varepsilon(x)\Li(x)+\frac{C_0 n}{\log^2 n},
\]
where we used $|E(n-1;p,-2)|(p-1)\le\varepsilon(x)\Li(x)$ from the good-modulus definition and $\Pi_2(n-1;p,-2)(p-1)\le C_0 n/(\log n)^2$ from Theorem~\ref{thm:twin-selberg}. We absorb the twin correction into the left side. Since $\Li(n-1)\ge n/(2\log n)$ for large $n$ and $\log n\ge\tfrac{1}{2}\log x$ in the range $n\in[x^{1/2},x]$,
\[
\frac{C_0 n}{\log^2 n}\le\frac{2C_0}{\log n}\,\Li(n-1)\le\frac{1}{2}\,\Li(n-1)
\]
for $x$ large enough. Rearranging gives
\[
\tfrac{1}{2}\,\Li(n-1)\le k(p-1)+\varepsilon(x)\Li(x).
\]
The right side is $O(x(\log x)^{-6})$ by the same estimates as in Lemma~\ref{lem:good}, so $\Li(n-1)\ll x(\log x)^{-6}$ and $n\ll x(\log x)^{-5}$. All such $n$ lie in an interval of length $o(x)$.

\emph{Large odd primes} ($p>Q$). Lemmas~\ref{lem:large-squares} and~\ref{lem:two-large} apply unchanged: outside $o(x)$ exceptions, $n$ has at most one large odd prime factor to the first power, contributing odd deficit at most $1$.

\emph{The prime $2$.} By Proposition~\ref{prop:staircase}, whenever $c_n$ is even one has $c_n=2$. Thus the case $c_n=2\ell$ with $\ell$ an odd prime cannot occur, and the even part never creates a composite obstruction.

\emph{Assembly.} For all but $o(x)$ integers $n\le x$, the small and medium odd-prime deficits vanish by the preceding estimates, and the large odd-prime deficit is at most $1$. By Corollary~\ref{cor:odd-deficit-k2}, the odd part of $c_n$ is $1$ or an odd prime. Since, by Proposition~\ref{prop:staircase}, whenever $c_n$ is even one has $c_n=2$, it follows that $c_n\in\{1\}\cup\PP$.
\end{proof}

Thus, for $K=2$, a prime-index description together with a sparse twin-loss bound suffices to recover the same density-$1$ conclusion as in the unconditional $K=1$ case.

\begin{remark}\label{rem:twin-finite}
Under Hypothesis~\ref{hyp:detect}, if there are only finitely many twin primes, then $c_q=q$ for all sufficiently large primes $q$, the stock becomes complete beyond some threshold, and the companion-sieve operates exactly as in the $K=1$ case without any twin correction. Conversely, if there are infinitely many twins, the stock has infinitely many gaps, but Theorem~\ref{thm:twin-selberg} ensures these gaps are too sparse to obstruct a density-$1$ argument. In either case, the companion-sieve yields the same density-$1$ conclusion.
\end{remark}

\section{Beyond density $1$}\label{sec:conclusion}

\subsection{Summary of results}

We have proved unconditionally that the LCM prime-increment conjecture holds on a set of density~$1$ (Theorem~\ref{thm:density1}), with an effective finite reduction eliminating all small prime obstructions (Proposition~\ref{prop:Nstar}). For the $K=2$ variant, the sequence $(c_n)$ detects twin primes (Proposition~\ref{prop:detection}), exhibits a $2$-adic staircase (Proposition~\ref{prop:staircase}), constrains the reappearance of larger twin primes via a valuation barrier (Proposition~\ref{prop:barrier-twin}), and admits a density-$1$ theorem conditional on a prime-index detection hypothesis, with the twin-prime loss controlled unconditionally by the upper-bound sieve estimate of Theorem~\ref{thm:twin-selberg}; this yields the conditional density-$1$ result of Theorem~\ref{thm:density-k2}.

The residual obstruction to proving the conjecture for all large $n$ decomposes into three cases. Case (C1) consists of integers with a squared large prime factor. Case (C2) consists of integers with two distinct large prime factors. Case (C3) consists of integers divisible by a medium prime that is individually bad for the distribution of primes in the class $-1\bmod p$.

\subsection{A conditional result under Elliott--Halberstam}

Assume the Elliott--Halberstam conjecture \cite{EH68} at some level $\vartheta>1/2$, namely that for every $A>0$,
\[
\sum_{q\le x^\vartheta}\max_{\gcd(a,q)=1}\max_{y\le x}|E(y;q,a)|\ll_A\frac{x}{(\log x)^A}.
\]

\begin{proposition}\label{prop:EH}
Assume Elliott--Halberstam at level $\vartheta>1/2$. If one reruns the density argument of Section~\ref{sec:density} with threshold $Q_{\rm EH}=x^\vartheta/(\log x)^B$ in place of $Q=x^{1/2}/(\log x)^B$, then the large-prime cases disappear. No integer $n\le x$ can satisfy $p^2\mid n$ with $p>Q_{\rm EH}$, nor $p_1p_2\mid n$ with distinct $p_1,p_2>Q_{\rm EH}$.
\end{proposition}
\begin{proof}
If $p>Q_{\rm EH}$, then $p^2>x^{2\vartheta}/(\log x)^{2B}>x$ for large $x$, since $2\vartheta>1$. Hence $p^2\nmid n$ for $n\le x$. Likewise, $p_1p_2>x$ for distinct $p_1,p_2>Q_{\rm EH}$.
\end{proof}

Under Elliott--Halberstam, after replacing the Bombieri--Vinogradov threshold by $Q_{\rm EH}$, the only remaining obstruction is the analogue of case~(C3) in the enlarged medium range. A specific prime divisor of~$n$ may still be individually exceptional. This is controlled on average, but not pointwise.

Recent work of Maynard~\cite{Maynard25} establishes unconditionally, for a fixed residue class, mean value theorems for primes in progressions to moduli of size $x^{1/2+\delta}$ with $\delta>0$, extending Bombieri--Vinogradov beyond the classical $x^{1/2}$ barrier up to the exponent $11/21$. This would in principle allow the analysis of Section~\ref{sec:density} to be rerun with a slightly larger $Q$, further thinning the large-prime exceptional set. Since the companion-sieve argument already gives density~$1$ unconditionally, the gain would be in the effective threshold $N^*(P_0)$ rather than in the asymptotic statement.

\subsection{Three conjectures}

The analysis identifies three conjectures. The first and third address directly the residual obstruction in the companion-sieve analysis. The second, motivated by the Chebyshev bias visible in the numerical data for the first, is an independent statement on the distribution of the leading exceptions. All three are supported by extensive computation and sit strictly between current unconditional technology and the pointwise estimates predicted by Montgomery.

The first conjecture asks that, at the square window $q=p^2$, the count of primes in the progression $-1\bmod p$ never drops below two.

\begin{conjecture}\label{conj:SQ2}
For every prime $p\ge17$,
\[
\picount(p^2;\,p,\,-1)\ge 2.
\]
\end{conjecture}

In other words, there are at least two primes $q\equiv -1\pmod{p}$ below $p^2$. The expected count is $\picount(p^2;p,-1)\sim p/(2\log p)\to\infty$, so the conjecture is consistent with the prime number theorem in arithmetic progressions, but it lies at the boundary $q=x^{1/2}$, beyond the reach of the standard pointwise consequences of the Generalized Riemann Hypothesis and outside the scope of Bombieri--Vinogradov. It may be viewed as a second-prime analogue of Linnik's problem in the class $-1\bmod p$. Computation up to $p=20000$ shows that only $p=5$ and $p=13$ fail, and $\picount(p^2;p,-1)\ge 3$ for all $p\ge 17$ in this range. The ratio $R_p:=\picount(p^2;p,-1)\big/\bigl(p/(2\log p)\bigr)$ has mean $\approx 1.07$ over $3\le p\le 1000$, confirming close agreement with the heuristic.

A probabilistic heuristic clarifies why the conjecture is expected to hold.

\begin{remark}\label{rem:cramer}
A Cram\'er-type model for primes in the progression $-1\bmod p$ predicts that $\picount(p^2;p,-1)$ is approximately Poisson with mean $\lambda_p=p/(2\log p)$, so that the probability of failure $\picount(p^2;p,-1)\le 1$ decays as $e^{-\lambda_p}(1+\lambda_p)$, summable over primes. A Borel--Cantelli argument under this model would predict only finitely many exceptional $p$, consistent with the observed behavior. The heuristic must however be applied with care at the boundary $q=x^{1/2}$, where Maier's phenomenon \cite{Maier85} produces fluctuations beyond the Poisson scale, and where Chebyshev's bias in the sense of Rubinstein--Sarnak \cite{RubinsteinSarnak94} acts on the progression $-1\bmod p$: the class $-1$ is a quadratic residue modulo $p$ exactly when $p\equiv 1\pmod 4$, in which case primes tend to be slightly underrepresented. A direct computation for $3\le p\le 1000$ gives mean ratios $R_p\approx 1.04$ on the $80$ primes $p\equiv 1\pmod 4$ and $R_p\approx 1.09$ on the $87$ primes $p\equiv 3\pmod 4$, and the two failures of Conjecture~\ref{conj:SQ2} both lie in the first class.
\end{remark}

The observation at the end of Remark~\ref{rem:cramer} motivates a second conjecture, isolating the Chebyshev asymmetry in the relevant window.

\begin{conjecture}\label{conj:CB}
As $P\to\infty$,
\[
\frac{1}{\pi(P;4,1)}\sum_{\substack{3\le p\le P\\ p\equiv 1\,(4)}}R_p
\;<\;
\frac{1}{\pi(P;4,3)}\sum_{\substack{3\le p\le P\\ p\equiv 3\,(4)}}R_p,
\]
and both averages tend to $1$. Equivalently, the deficit is a lower-order Chebyshev correction, not a departure from the leading asymptotic.
\end{conjecture}

Conjecture~\ref{conj:CB} is a quantitative strengthening consistent with the Rubinstein--Sarnak framework \cite{RubinsteinSarnak94} applied to the specific height $x=p^2$. It is independent of Conjecture~\ref{conj:SQ2} but provides a conceptual explanation for why the (finitely many) failures of that conjecture concentrate in the class $p\equiv 1\pmod 4$.

If Conjecture~\ref{conj:SQ2} holds, then for every large prime $p$ and every $n$ with $p^2\mid n$, the $1$-hit stock satisfies $S_p^{(1)}(n-1)\ge 2$, giving $\delta_p^{(1)}(n)=0$ in the large-prime regime. This eliminates case~(C1).

The third conjecture addresses case~(C3) and asks for pointwise control of the Dirichlet error term at prime divisors of $n$, up to a level $\vartheta$.

\begin{conjecture}\label{conj:DS}
Let $\vartheta\in(0,1)$. For every $A>0$, there exists $B>0$ such that for all sufficiently large integers $n$ and every prime $p\mid n$ with $p\le n^\vartheta/(\log n)^B$,
\[
\bigl|E(n-1;\,p,\,-1)\bigr|\le\frac{\Li(n)}{(p-1)(\log n)^A}.
\]
\end{conjecture}

This asks that the prime divisors of $n$ up to level $n^\vartheta$ are never individually bad moduli. At $\vartheta=1/2$, it is stronger in nature than Bombieri--Vinogradov, since it is pointwise on divisors of $n$ rather than averaged over moduli, but strictly weaker than Montgomery's pointwise conjecture \cite{Montgomery71} $E(x;q,a)\ll_\varepsilon x^{1/2+\varepsilon}q^{-1/2}$. It is not implied by the Elliott--Halberstam conjecture in its standard averaged form. Numerical evidence up to $n=10^6$ shows rapid decay of the normalized error at prime divisors in the medium range.

If Conjecture~\ref{conj:DS} holds at level $\vartheta=1/2$, the good-modulus argument of Lemma~\ref{lem:good} applies to every medium prime divisor of $n$, eliminating case~(C3) in the original Bombieri--Vinogradov range.

\medskip

Conjectures~\ref{conj:SQ2} and~\ref{conj:DS} address two distinct pieces of the residual obstruction: the former removes the square case~(C1), the latter removes the bad-medium-modulus case~(C3). The remaining unresolved case is then~(C2), involving two distinct large prime factors. If one further assumes Elliott--Halberstam at some level $\vartheta>1/2$ and reruns the argument with threshold $Q_{\rm EH}=x^\vartheta/(\log x)^B$, then the large-prime cases~(C1) and~(C2) both disappear (Proposition~\ref{prop:EH}), so Conjecture~\ref{conj:DS} at level~$\vartheta$ would imply Conjecture~\ref{conj:main} for all sufficiently large~$n$.

\subsection{Further questions}

Is $t_0(p)\le 3$ for all primes? Can the larger twin primes be completely excluded from the sequence $(c_n)$, beyond the $q^2$~barrier of Proposition~\ref{prop:barrier-twin}? Do the $K=3,4,\ldots$ variants of the recurrence encode other prime-gap conjectures?


\end{document}